\documentclass[12pt]{amsart}

\usepackage{amsmath}
\usepackage{amssymb}
\usepackage{amsfonts}
\usepackage{amsthm}
\usepackage{enumerate}
\usepackage{hyperref}
\usepackage{color}

\theoremstyle{plain}
\newtheorem{thm}{Theorem}[section]

\newtheorem{lem}[thm]{Lemma}
\newtheorem{cor}[thm]{Corollary}

\newtheorem{conj}[thm]{Conjecture}

\theoremstyle{definition}
\newtheorem{dfn}[thm]{Definition}

\newtheorem{exmps}[thm]{Examples}
\newtheorem{rem}[thm]{Remark}

\newtheorem{dfns-rems}[thm]{Definitions and Remarks}
\newtheorem{notas-rems}[thm]{Notations and Remarks}
\newtheorem{exmps-rems}[thm]{Examples and Remarks}

\begin{document}


\title[Depth, sdepth and regularity]{Depth, Stanley depth and regularity of ideals associated to graphs}


\author[S. A. Seyed Fakhari]{S. A. Seyed Fakhari}

\address{S. A. Seyed Fakhari, School of Mathematics, Statistics and Computer Science,
College of Science, University of Tehran, Tehran, Iran.}

\email{fakhari@khayam.ut.ac.ir}

\urladdr{http://math.ipm.ac.ir/$\sim$fakhari/}


\begin{abstract}
Let $\mathbb{K}$ be a field and $S=\mathbb{K}[x_1,\dots,x_n]$ be the
polynomial ring in $n$ variables over $\mathbb{K}$. Let $G$ be a
graph with $n$ vertices. Assume that $I=I(G)$ is the edge ideal of $G$ and $J=J(G)$ is its cover ideal. We prove that ${\rm sdepth}(J)\geq n-\nu_{o}(G)$ and ${\rm sdepth}(S/J)\geq n-\nu_{o}(G)-1$, where $\nu_{o}(G)$ is the ordered matching number of $G$. We also prove the inequalities ${\rm sdepth}(J^k)\geq {\rm depth}(J^k)$ and ${\rm sdepth}(S/J^k)\geq {\rm depth}(S/J^k)$, for every integer $k\gg 0$, when $G$ is a bipartite graph. Moreover, we provide an elementary proof for the known inequality ${\rm reg}(S/I)\leq \nu_{o}(G)$.
\end{abstract}






\thanks{}


\maketitle


\section{Introduction and Preliminaries} \label{sec1}

Let $\mathbb{K}$ be a field and let $S=\mathbb{K}[x_1,\dots,x_n]$
be the polynomial ring in $n$ variables over $\mathbb{K}$. Let
$M$ be a finitely generated $\mathbb{Z}^n$-graded $S$-module. Let
$u\in M$ be a homogeneous element and $Z\subseteq
\{x_1,\dots,x_n\}$. The $\mathbb {K}$-subspace $u\mathbb{K}[Z]$
generated by all elements $uv$ with $v\in \mathbb{K}[Z]$ is
called a {\it Stanley space} of dimension $|Z|$, if it is a free
$\mathbb{K}[Z]$-module. Here, as usual, $|Z|$ denotes the number
of elements of $Z$. A decomposition $\mathcal{D}$ of $M$ as a
finite direct sum of Stanley spaces is called a {\it Stanley
decomposition} of $M$. The minimum dimension of a Stanley space
in $\mathcal{D}$ is called the {\it Stanley depth} of
$\mathcal{D}$ and is denoted by ${\rm sdepth} (\mathcal {D})$.
The quantity $${\rm sdepth}(M):=\max\big\{{\rm sdepth}
(\mathcal{D})\mid \mathcal{D}\ {\rm is\ a\ Stanley\
decomposition\ of}\ M\big\}$$ is called the {\it Stanley depth}
of $M$. We say that a $\mathbb{Z}^n$-graded $S$-module $M$ satisfies {\it Stanley's inequality} if $${\rm depth}(M) \leq
{\rm sdepth}(M).$$ In fact, Stanley \cite{s} conjectured that every $\mathbb{Z}^n$-graded $S$-module satisfies Stanley's inequality.
This conjecture has been recently disproved in \cite{abcj}.
However, it is still interesting to find the classes of
$\mathbb{Z}^n$-graded $S$-modules which satisfy Stanley's inequality.
For a reader friendly introduction to Stanley depth, we refer to
\cite{psty} and for a nice survey on this topic, we refer to
\cite{h}.

Let $G$ be a graph with vertex set $V(G)=\big\{x_1, \ldots,
x_n\big\}$ and edge set $E(G)$ (by abusing the notation, we identify the vertices of $G$ with the variables of $S$). For a vertex $x_i$, the {\it neighbor set} of $x_i$ is $N_G(x_i)=\{x_j\mid x_ix_j\in E(G)\}$ and We set $N_G[x_i]=N_G(x_i)\cup \{x_i\}$ and call it the {\it closed neighborhood} of $x_i$. For every subset $A\subset V(G)$, the graph $G\setminus A$ is the graph with vertex set $V(G\setminus A)=V(G)\setminus A$ and edge set $E(G\setminus A)=\{e\in E(G)\mid e\cap A=\emptyset\}$. A {\it bipartite graph} is one whose vertex set is partitioned into two (not necessarily nonempty) disjoint subsets in such a way that the two end vertices for each edge lie in distinct partitions. A {\it matching} in a graph is a set of edges such that no two different edges share a common vertex. A subset $W$ of $V(G)$ is called an {\it independent subset} of $G$ if there are no edges among the vertices of $W$. A subset $C$ of $V(G)$ is called a {\it vertex cover} of the graph $G$ if every edge of $G$ is incident to at least one vertex of $C$. A vertex cover $C$ is called a {\it minimal vertex cover} of $G$ if no proper subset of $C$ is a vertex cover of $G$.

Next, we define the notion of ordered matching for a graph. It was introduced in \cite{cv} and plays a central role in this paper.

\begin{dfn} \label{om}
Let $G$ be a graph, and let $M=\{\{a_i,b_i\}\mid 1\leq i\leq r\}$ be a
nonempty matching of $G$. We say that $M$ is an {\it ordered matching} of
$G$ if the following hold:
\begin{itemize}
\item[(1)] $A:=\{a_1,\ldots,a_r\} \subseteq V(G)$ is a set of
    independent vertices of $G$; and

\item[(2)] $\{a_i, b_j\}\in E(G)$ implies that $i\leq j$.
\end{itemize}
The {\it ordered matching number} of $G$, denoted by $\nu_{o}(G)$, is
defined to be $$\nu_{o}(G)=\max\{|M|\mid M\subseteq E(G)\ {\rm is\ an\
ordered\ matching\ of} \ G\}.$$
\end{dfn}

The edge ideal $I(G)$ of $G$ is the ideal of $S$ generated by the squarefree  monomials  $x_ix_j$, where $\{x_i, x_j\}$ is an edge of $G$. The Alexander dual of the edge ideal of $G$ in $S$, i.e., the
ideal $$J(G)=I(G)^{\vee}=\bigcap_{\{x_i,x_j\}\in E(G)}(x_i,x_j),$$ is called the
{\it cover ideal} of $G$ in $S$. The reason for this name is due to the
well-known fact that the generators of $J(G)$ correspond to minimal vertex covers of $G$.

The main goal of This paper is to study the Stanley depth of cover ideals and their power. In Theorem \ref{cover}, we prove that for every graph $G$, the inequalities ${\rm sdepth}(J(G))\geq n-\nu_{o}(G)$ and ${\rm sdepth}(S/J(G))\geq
n-\nu_{o}(G)-1$ hold. In that theorem, we also prove that the same inequalities hold, if one replaces sdepth by depth. Then, in Corollary \ref{reg}, we conclude that for every graph $G$ we have ${\rm reg}(S/I)\leq \nu_{o}(G)$. This inequality was previously proved by Constantinescu and Varbaro \cite[Remark 4.8]{cv}. However, our proof is more elementary.

In Section \ref{sec3}, we consider the Stanley depth of powers of cover ideal of bipartite graphs. Let $G$ be a bipartite graph. In \cite[Corollary 3.6]{s3}, the author proved that the sequences $\{{\rm sdepth}(J(G)^k)\}_{k=1}^{\infty}$ and $\{{\rm sdepth}(S/J(G)^k)\}_{k=1}^{\infty}$ are non-increasing. Thus the both sequences are convergent. In Theorem \ref{main}, we provide lower bounds for the limit value of theses sequences. Indeed, we prove that for every bipartite graph $G$, we have$$\lim_{k\to\infty}{\rm sdepth}(J(G)^k)\geq n-\nu_{o}(G) \ \ \ {\rm and} \ \ \ \lim_{k\to\infty}{\rm sdepth}(S/J(G)^k)\geq n-\nu_{o}(G)-1.$$Then we conclude in Corollary \ref{sin} that $J(G)^k$ and $S/J(G)^k$ satisfy the Stanley's inequality, for every integer $k\gg 0$. Theorem \ref{main} also shows that a conjecture of the author is true for the powers of cover ideal of bipartite graphs (see Conjecture \ref{conje} and the paragraph after it).


\section{First Power} \label{sec2}

The first main result of this paper is Theorem \ref{cover}, which provides a lower bound for the depth and the Stanley depth of cover ideal of graphs. We first need the following three simple lemmas. The first one shows that the ordered matching number of a graph strictly decreases when we delete the closed neighborhood of a non-isolated vertex.

\begin{lem} \label{lom}
Let $G$ be a graph and $x$ be a non-isolated vertex of $G$. Then we have $\nu_{o}(G\setminus N_G[x])\leq \nu_{o}(G)-1$.
\end{lem}

\begin{proof}
Assume that $\nu_{o}(G\setminus N_G[x])=t$ and let $M=\{\{a_i,b_i\}\mid 1\leq i\leq t\}$ be an ordered matching of $G\setminus N_G[x]$. Since $x$ is not isolated, we may choose a vertex $y\in N_G(x)$. Set $a_{t+1}=x$ and $b_{t+1}=y$. Then $\{a_1,\ldots,a_{t+1}\}$ is a set of  independent vertices of $G$, because $a_1, \ldots, a_t$ are vertices of $G\setminus N_G[x]$. By the same reason, $a_{t+1}$ is not adjacent to $b_1, \ldots, b_t$. This shows that $M\cup \{a_{t+1}, b_{t+1}\}$ is an ordered matching of $G$ and therefore, $\nu_{o}(G)\geq t+1$.
\end{proof}

The next Lemma shows that how the cover ideal of a graph $G$ and that of $G\setminus N_G[x]$ are related, when $x$ is an arbitrary vertex of $G$.

\begin{lem} \label{del}
Let $G$ ba a graph with vertex set $V(G)=\{x_1, \ldots, x_n\}$. Assume that $x\in V(G)$ is a vertex of $G$. Set $u=\prod_{x_i\in N_G(x)}x_i$ and $J'=J(G\setminus N_G[x])S$. Then $J(G)+(x)=uJ'+(x)$.
\end{lem}

\begin{proof}
Let $C$ be a vertex cover of $G$ with $x\notin C$. Then $N_G(x)\subseteq C$ and $C\setminus N_G(x)$ is a vertex cover of $G\setminus N_G[x]$. This shows that $J(G)+(x)\subseteq uJ'+(x)$. For the converse inclusion, assume that $D$ is a vertex cover of $G\setminus N_G[x]$. Then $D\cup N_G(x)$ is a vertex cover of $G$. This shows that $uJ'+(x)\subseteq J(G)+(x)$ and completes the proof.
\end{proof}

The following lemma provides a combinatorial description for the colon of cover ideals.

\begin{lem} \label{colon}
Let $G$ ba a graph with vertex set $V(G)=\{x_1, \ldots, x_n\}$. Assume that $x\in V(G)$ is a vertex of $G$. Set $J'=J(G\setminus x)S$. Then $(J(G): x)=J'$.
\end{lem}

\begin{proof}
If $C$ is a vertex cover of $G$, then $C\setminus \{x\}$ is a vertex cover of $G\setminus x$. This shows that $(J(G): x)\subseteq J'$. On the other hand, if $D$ is a vertex cover of $G\setminus x$, then $D\cup \{x\}$ is a vertex cover of $G$. This shows that $J'\subseteq (J(G):x)$.
\end{proof}

We are now ready to prove the first main result of this paper. As we mentioned in introduction, the second part of this theorem is known by \cite[Remark 4.8]{cv}. But our argument is completely different and provides a simple proof for it.

\begin{thm} \label{cover}
Let $G$ be a graph and $J(G)$ be its cover ideal. Then

\begin{itemize}
\item[(i)] ${\rm sdepth}(J(G))\geq n-\nu_{o}(G)$ and ${\rm sdepth}(S/J(G))\geq
    n-\nu_{o}(G)-1$,\\[-0.3cm]
\item[(ii)] ${\rm depth}(S/J(G))\geq n-\nu_{o}(G)-1$.
\end{itemize}
\end{thm}

\begin{proof}
We prove (i) and (ii) simultaneously by induction on the number of edges of $G$. If $G$ has only one edge, then $\nu_{o}(G)=1$ and $J(G)$ is generated by two variables. Then ${\rm
depth}(S/J(G))=n-2$. Also, ${\rm sdepth}(S/J(G))=n-2$ by \cite[Theorem
1.1]{r} and ${\rm sdepth}(J(G))\geq n-1$ by \cite[Corollary 24]{h} and \cite[Lemma 3.6]{hvz}. Therefore, in
these cases, the inequalities in (i) and (ii) are trivial.

We now assume that $G$ has at least two edges. Note that, $G$ has at least one non-isolated vertex. Without loss of generality, we may assume that $x_1$ is a non-isolated vertex of $G$. Let $S'=\mathbb{K}[x_2, \ldots, x_n]$ be the polynomial ring obtained from $S$ by deleting the variable $x_1$ and consider the ideals $J'=J(G)\cap S'$ and
$J''=(J(G):x_1)$.

Now $J(G)=J'S'\oplus x_1J''S$ and $S/J(G)=(S'/J'S')\oplus
x_1(S/J''S)$ (as vector spaces) and therefore by definition of  the Stanley depth we have
\[
\begin{array}{rl}
{\rm sdepth}(J(G))\geq \min \{{\rm sdepth}_{S'}(J'S'), {\rm sdepth}_S(J'')\},
\end{array} \tag{1} \label{1}
\]
and
\[
\begin{array}{rl}
{\rm sdepth}(S/J(G))\geq \min \{{\rm sdepth}_{S'}(S'/J'S'), {\rm sdepth}_S(S/J'')\}.
\end{array} \tag{2} \label{2}
\]
On the other hand, by applying the depth lemma on the exact sequence
\[
\begin{array}{rl}
0\longrightarrow S/(J(G):x_1)\longrightarrow S/J(G)\longrightarrow S/(J(G), x_1)
\longrightarrow 0
\end{array}
\]
we conclude that
\[
\begin{array}{rl}
{\rm depth}(S/J(G))\geq \min \{{\rm depth}_{S'}(S'/J'S'), {\rm depth}_S(S/J'')\}.
\end{array} \tag{3} \label{3}
\]
Using Lemma \ref{colon}, it follows that $J''=J(G\setminus x_1)S$. Hence our induction hypothesis implies that $${\rm depth}_S(S/J'')={\rm depth}_{S'}(S'/J'')+1\geq n-1-\nu_{o}(G\setminus x_1)-1+1\geq n-\nu_{o}(G)-1.$$Also, it follows from \cite[Lemma 3.6]{hvz} that$${\rm sdepth}_S(S/J'')={\rm sdepth}_{S'}(S'/J'')+1\geq n-1-\nu_{o}(G\setminus x_1)-1+1\geq n-\nu_{o}(G)-1,$$ and $${\rm sdepth}_S(J'')={\rm sdepth}_{S'}(J'')+1\geq n-1-\nu_{o}(G\setminus x_1)+1\geq n-\nu_{o}(G).$$

On the other hand, it follows from Lemma \ref{del} that there exists a monomial $u\in S'$ such that $J'S'=uJ(G\setminus N_G[x_1])S'$. Since $uJ(G\setminus N_G[x_1])S'$ and $J(G\setminus N_G[x_1])S'$, (up to a shift) are isomorphic as graded $S'$-Modules, we conclude that ${\rm depth}_{S'}(J'S')={\rm depth}_{S'}(J(G\setminus N_G[x_1])S')$. On the other hand, it follows from \cite[Theorem 1.1]{c1} that ${\rm sdepth}_{S'}(J'S')={\rm sdepth}_{S'}(J(G\setminus N_G[x_1])S')$ and ${\rm sdepth}_{S'}(S'/J'S')={\rm sdepth}_{S'}(S'/J(G\setminus N_G[x_1])S')$. Therefore by \cite[Lemma 3.6]{hvz}, Lemma \ref{lom} and
the induction hypothesis we conclude that

$${\rm sdepth}_{S'}(J'S')={\rm sdepth}_{S'}(J(G\setminus N_G[x_1])S')\geq n-1-\nu_{o}(G\setminus N_G[x_1])\geq n-\nu_{o}(G),$$
and similarly ${\rm sdepth}_{S'}(S'/J'S')\geq n-\nu_{o}(G)-1$ and ${\rm
depth}_{S'}(S'/J'S')\geq n-\nu_{o}(G)-1$. Now the assertions follow by inequalities (\ref{1}), (\ref{2})
and (\ref{3}).
\end{proof}

Let $M$ be a finitely generated graded $S$-Module. The {\it Castelnuovo-Mumford regularity} (or simply, regularity) of $M$, denoted by ${\rm reg}(M)$, is defined as follows:
$${\rm reg}(M)=\max\{j-i|\ {\rm Tor}_i^S(\mathbb{K}, M)_j\neq0\}.$$
The regularity of a module is one of the most important homological invariants of it. Computing the regularity of edge ideals or finding bounds for it has been studied by a number of researchers (see for example \cite{dhs}, \cite{hv}, \cite{k}, \cite{n}, \cite{va}).

An immediate consequence of the second part of theorem \ref{cover} is the following corollary.

\begin{cor} \label{reg}
For every graph $G$, we have ${\rm reg}(S/I(G))\leq \nu_{o}(G)$.
\end{cor}

\begin{proof}
It follows from Theorem \ref{cover} and the Auslander-Buchsbaum Formula that the projective dimension of $J(G)$ is at most $\nu_{o}(G)$. Then it follows from Terai's theorem \cite[Theorem 8.1.10]{hh} that ${\rm reg}(S/I(G))\leq \nu_{o}(G)$.
\end{proof}

\begin{rem}
One can give a direct proof for the above corollary, by applying \cite[Corollary 18.7]{p'} on the following exact sequence.$$0\longrightarrow S/(I(G):x_1)\longrightarrow S/I(G)\longrightarrow S/(I(G), x_1)\longrightarrow  0$$However, this proof is essentially the same as given above.
\end{rem}

In \cite{hv}, H${\rm \grave{a}}$ and Van Tuyl proved that the for every graph $G$, the regularity of $S/I(G)$ is less than or equal to the maximum cardinality of matchings of $G$. In fact, it follows
from their proof (and was explicitly stated in \cite{w}) that the ${\rm reg}(S/I(G))$ is at most the minimum cardinality of maximal
matchings of $G$. The following examples show that this bound is not comparable with the bond given in Corollary \ref{reg}.

\begin{exmps}
\begin{enumerate}
\item Let $G=C_4$ be the $4$-cycle-graph. Then one can easily check that $\nu_{o}(G)=1$ and the cardinality of every maximal
matching of $G$ is equal to $2$. Thus, in this example, $\nu_{o}(G)$ is strictly less than the minimum cardinality of maximal
matchings of $G$. We also have ${\rm reg}(S/I(G))=1=\nu_{o}(G)$.
\item Let $G=P_4$ be the path with $4$ vertices. Then one can easily check that $\nu_{o}(G)=2$, while the minimum cardinality of maximal
matchings of $G$ is equal to $1$. Thus, in this example, the minimum cardinality of maximal
matchings of $G$ is strictly less than $\nu_{o}(G)$. We also have ${\rm reg}(S/I(G))=1$ is equal to the minimum cardinality of maximal
matchings of $G$.
\end{enumerate}
\end{exmps}


\section{High Powers} \label{sec3}

The aim of this section is to prove that the high powers of cover ideal of bipartite graphs satisfy the Stanley's inequality. To do this, in Theorem \ref{main}, we provide a lower bound for the Stanley depth of cover ideal of bipartite graphs. Before that, in Lemma \ref{bcolon}, we prove that the different powers of cover ideal of a bipartite graphs, can be obtained from each other by taking colon with respect to a suitable monomial. To prove Lemma \ref{bcolon}, we need to remind the definition of symbolic powers.

\begin{dfn}
Let $I$ be a squarefree monomial ideal in $S$ and suppose that $I$ has the irredundant
primary decomposition $$I=\frak{p}_1\cap\ldots\cap\frak{p}_r,$$ where every
$\frak{p}_i$ is an ideal of $S$ generated by a subset of the variables of
$S$. Let $k$ be a positive integer. The $k$th {\it symbolic power} of $I$,
denoted by $I^{(k)}$, is defined to be $$I^{(k)}=\frak{p}_1^k\cap\ldots\cap
\frak{p}_r^k.$$
\end{dfn}

The proof of the following lemma is based on the fact that the symbolic and the ordinary powers of cover ideal of bipartite graphs coincide.

\begin{lem} \label{bcolon}
Let $G$ be a bipartite graph and assume that $V(G)=U\cup W$ is a bipartition for the vertex set of $G$. Set $u=\prod_{x_i\in U}x_i$. Then for every integer $k\geq 1$, we have $(J(G)^k:u)=J(G)^{k-1}$.
\end{lem}

\begin{proof}
It follows from \cite[Corollary 2.6]{grv} that for every integer $k\geq 1$ we have $J(G)^k=J(G)^{(k)}$. On the other hand, for every edge $e=\{x_i, x_j\}$ of $G$, we have $\mid e\cap U\mid=1$. Thus, $((x_i,x_j)^k:u)=(x_i,x_j)^{k-1}$, for every integer $k\geq 1$. Hence$$(J(G)^k:u)=(J(G)^{(k)}:u)=\bigcap_{\{x_i, x_j\}\in E(G)}((x_i,x_j)^k:u)$$
$$=\bigcap_{\{x_i, x_j\}\in E(G)}(x_i,x_j)^{k-1}=J(G)^{(k-1)}=J(G)^{k-1}.$$
\end{proof}

As we mentioned in the the first section, the sequences $\{{\rm sdepth}(J(G)^k)\}_{k=1}^{\infty}$ and $\{{\rm sdepth}(S/J(G)^k)\}_{k=1}^{\infty}$ are convergent. In the following theorem, we provide lower bounds for the limit of theses sequences.

\begin{thm} \label{main}
Let $G$ be a bipartite graph. Then for every integer $k\geq 1$, the inequalities$${\rm sdepth}(J(G)^k)\geq n-\nu_{o}(G) \ \ \ \ \ {\rm and} \ \ \ \ \ {\rm sdepth}(S/J(G)^k)\geq n-\nu_{o}(G)-1$$hold.
\end{thm}

\begin{proof}
Assume that $V(G)=U\cup W$ is a bipartition for the vertex set of $G$. Without loss of generality, we may assume that $U=\{x_1, \ldots, x_t\}$ and $W=\{x_{t+1}, \ldots, x_n\}$, for some integer $t$ with $1\leq t\leq n$. Let $m$ be the number of edges of $G$. We prove the assertions by induction on $m+k$. First, we can assume that $G$ has no isolated vertex. Because deleting the isolated vertices does not change the cover ideal and the ordered matching number of $G$.

For $k=1$, the assertions follow from Theorem \ref{cover}. If $m=1$, then $G$ has two vertices and $\nu_{o}(G)=1$. In this case, the first inequality follows from \cite[Corollary 24]{h} and the second inequality is trivial. Therefore, assume that $k,m\geq 2$. Let $S_1=\mathbb{K}[x_2, \ldots, x_n]$ be the polynomial ring obtained from $S$ by deleting the variable $x_1$ and consider the ideals $J_1=J(G)^k\cap S_1$ and
$J_1'=(J(G)^k:x_1)$.

Now $J(G)^k=J_1\oplus x_1J_1'$ and $S/J(G)^k=(S_1/J_1)\oplus
x_1(S/J_1')$ (as vector spaces) and therefore by definition of  the Stanley depth we have
\[
\begin{array}{rl}
{\rm sdepth}(J(G)^k)\geq \min \{{\rm sdepth}_{S_1}(J_1), {\rm sdepth}_S(J_1')\},
\end{array} \tag{\dag} \label{dag}
\]
and
\[
\begin{array}{rl}
{\rm sdepth}(S/J(G)^k)\geq \min \{{\rm sdepth}_{S_1}(S_1/J_1), {\rm sdepth}_S(S/J_1')\}.
\end{array} \tag{\ddag} \label{ddag}
\]

Notice that $J_1=(J(G)\cap S_1)^k$. Hence, by Lemma \ref{del} we conclude that there exists a monomial $u_1\in S_1$ such that $J_1=u_1^kJ(G\setminus N_G[x_1])^kS_1$. It follows from \cite[Theorem 1.1]{c1} that ${\rm sdepth}_{S_1}(J_1)={\rm sdepth}_{S_1}(J(G\setminus N_G[x_1])^kS_1)$ and ${\rm sdepth}_{S_1}(S_1/J_1S_1)={\rm sdepth}_{S_1}(S_1/J(G\setminus N_G[x_1])^kS_1)$. Therefore, by \cite[Lemma 3.6]{hvz}, Lemma \ref{lom} and
the induction hypothesis, we conclude that
$${\rm sdepth}_{S_1}(J_1)={\rm sdepth}_{S_1}(J(G\setminus N_G[x_1])^kS_1)\geq n-1-\nu_{o}(G\setminus N_G[x_1])\geq n-\nu_{o}(G),$$
and similarly ${\rm sdepth}_{S_1}(S_1/J_1)\geq n-\nu_{o}(G)-1$. Thus, using the inequalities (\ref{dag}) and (\ref{ddag}), it is enough to prove that ${\rm sdepth}_S(J_1')\geq n-\nu_{o}(G)$ and ${\rm sdepth}_S(S/J_1')\geq n-\nu_{o}(G)-1$.

For every integer $i$ with $2\leq i\leq t$, let $S_i=\mathbb{K}[x_1, \ldots, x_{i-1}, x_{i+1}, \ldots, x_n]$ be the polynomial ring obtained from $S$ by deleting the variable $x_i$ and consider the ideals $J_i'=(J_{i-1}':x_i)$ and $J_i=J_{i-1}'\cap S_i$.

{\bf Claim.} For every integer $i$ with $1\leq i\leq t-1$ we have$${\rm sdepth}(J_i')\geq \min \{n-\nu_{o}(G), {\rm sdepth}(J_{i+1}')\}$$ and $${\rm sdepth}(S/J_i')\geq \min\{n-\nu_{o}(G)-1, {\rm sdepth}(S/J_{i+1}')\}.$$

{\it Proof of the Claim.} For every integer $i$ with $1\leq i\leq t-1$, we have $J_i'=J_{i+1}\oplus x_{i+1}J_{i+1}'$ and $S/J_i'=(S_{i+1}/J_{i+1})\oplus
x_{i+1}(S/J_{i+1}')$ (as vector spaces) and therefore by definition of  the Stanley depth we have
\[
\begin{array}{rl}
{\rm sdepth}(J_i')\geq \min \{{\rm sdepth}_{S_{i+1}}(J_{i+1}), {\rm sdepth}_S(J_{i+1}')\},
\end{array} \tag{$\ast$} \label{ast}
\]
and
\[
\begin{array}{rl}
{\rm sdepth}(S/J_i')\geq \min \{{\rm sdepth}_{S_{i+1}}(S_{i+1}/J_{i+1}), {\rm sdepth}_S(S/J_{i+1}')\}.
\end{array} \tag{$\ast\ast$} \label{aast}
\]

Notice that for every integer $i$ with $1\leq i\leq t-1$, we have $J_i'=(J(G)^k:x_1x_2\ldots x_i)$. Thus $J_{i+1}=J_i'\cap S_{i+1}=((J(G)^k\cap S_{i+1}):_{S_{i+1}}x_1x_2\ldots x_i)$. Hence, it follows from \cite[Proposition 2]{p} and \cite[Proposition 2.7]{c} (see also \cite[Proposition 2.5]{s3}) that

\[
\begin{array}{rl}
{\rm sdepth}_{S_{i+1}}(J_{i+1})\geq {\rm sdepth}_{S_{i+1}}(J(G)^k\cap S_{i+1}).
\end{array} \tag{$\ast\ast\ast$} \label{aaast}
\]

and

\[
\begin{array}{rl}
{\rm sdepth}_{S_{i+1}}(S_{i+1}/J_{i+1})\geq {\rm sdepth}_{S_{i+1}}(S_{i+1}/(J(G)^k\cap S_{i+1})).
\end{array} \tag{$\ast\ast\ast\ast$} \label{aaaast}
\]

By Lemma \ref{del} we conclude that there exists a monomial $u_{i+1}\in S_{i+1}$ such that $J(G)\cap S_{i+1}=u_{i+1}J(G\setminus N_G[x_{i+1}])S_{i+1}$. Therefore$$J(G)^k\cap S_{i+1}=u_{i+1}^kJ(G\setminus N_G[x_{i+1}])^kS_{i+1}$$ and it follows from \cite[Theorem 1.1]{c1} that$${\rm sdepth}_{S_{i+1}}(J(G)^k\cap S_{i+1})={\rm sdepth}_{S_{i+1}}(J(G\setminus N_G[x_{i+1}])^kS_{i+1})$$and$${\rm sdepth}_{S_{i+1}}(S_{i+1}/(J(G)^k\cap S_{i+1}))={\rm sdepth}_{S_{i+1}}(S_{i+1}/J(G\setminus N_G[x_{i+1}])^kS_{i+1}).$$ Therefore by \cite[Lemma 3.6]{hvz}, Lemma \ref{lom} and
the induction hypothesis we conclude that
$${\rm sdepth}_{S_{i+1}}(J(G)^k\cap S_{i+1})\geq n-1-\nu_{o}(G\setminus N_G[x_{i+1}])\geq n-\nu_{o}(G),$$
and similarly ${\rm sdepth}_{S_{i+1}}(S_{i+1}/(J(G)^k\cap S_{i+1}))\geq n-\nu_{o}(G)-1$. Now the claim follows by inequalities (\ref{ast}), (\ref{aast}), (\ref{aaast}) and (\ref{aaaast}).

Now, $J_t'=(J(G)^k:x_1x_2\ldots x_t)$ and hence, Lemma \ref{bcolon} implies that $J_t'=J(G)^{k-1}$ and thus, by induction hypothesis we conclude that ${\rm sdepth}(J_t')\geq n-\nu_{o}(G)$ and ${\rm sdepth}(S/J_t')\geq n-\nu_{o}(G)-1$. Therefore, using the claim repeatedly implies that ${\rm sdepth}(J_1')\geq n-\nu_{o}(G)$ and ${\rm sdepth}(S/J_1')\geq n-\nu_{o}(G)-1$. This completes the proof of the theorem.
\end{proof}

Let $I\subset S$ be a monomial ideal. A classical result by Burch \cite{b'} states
that $$\min_k{\rm depth}(S/I^k)\leq n-\ell(I),$$ where $\ell(I)$ is the
analytic spread of $I$, that is, the dimension of $\mathcal{R}(I)/
{{\frak{m}}\mathcal{R}(I)}$, where $\mathcal{R}(I)=\bigoplus_
{n=0}^{\infty}I^n= S[It] \subseteq S[t]$ is the Rees ring of $I$ and
$\frak{m}=(x_1,\ldots,x_n)$ is the maximal ideal of $S$. By a theorem of
Brodmann \cite{b}, ${\rm depth}(S/I^k)$ is constant for large $k$. We call
this constant value the {\it limit depth} of $I$, and denote it by
$\lim_{k\rightarrow \infty}{\rm depth}(S/I^k)$. Brodmann improved the Burch's
inequality by showing that

\[
\begin{array}{rl}
\lim_{k\rightarrow \infty}{\rm depth}(S/I^k) \leq n-\ell(I).
\end{array} \tag{$\sharp$} \label{sharp}
\]

Let $I\subset S$ be an arbitrary ideal. An element $f \in S$ is
{\it integral} over $I$, if there exists an equation
$$f^k + c_1f^{k-1}+ \ldots + c_{k-1}f + c_k = 0 {\rm \ \ \ \ with} \ c_i\in I^i.$$
The set of elements $\overline{I}$ in $S$ which are integral over $I$ is the {\it integral closure}
of $I$. The ideal $I$ is {\it integrally closed}, if $I = \overline{I}$.

It is known that the equality holds, in inequality (\ref{sharp}), if $I$ is a normal ideal. By \cite[Corollary 2.6]{grv} and \cite[Theorem 1.4.6]{hh}, we know that $J(G)$ is a normal ideal, for every bipartite graph $G$. Also, it follows from \cite[Theorem 2.8]{cv} that for every bipartite graph $G$, we have $\ell(J(G))=\nu_{o}(G)+1$. Thus, we conclude that$$\lim_{k\to\infty}{\rm
depth}(S/J(G)^k)=n-1-\nu_{o}(G).$$(This equality is explicitly stated in \cite[Theorem 4.5]{cpsty}.) Therefore, Theorem \ref{main} implies the following result

\begin{cor} \label{sin}
Let $G$ be a bipartite graph and $J(G)$ be its edge ideal. Then there exists an integer $n_0\geq 1$ such that $J(G)^k$ and $S/J(G)^k$ satisfy the Stanley's inequality, for every integer $k\geq n_0$.
\end{cor}

In \cite{s1}, the author proposed the following conjecture regarding the Stanley depth of integrally closed monomial ideals.

\begin{conj} \label{conje}
{\rm (}\cite[Conjecture 2.6]{s1}{\rm )} Let $I\subset S$ be an integrally closed monomial ideal. Then ${\rm sdepth}(S/I)\geq n-\ell(I)$ and ${\rm sdepth} (I)\geq n-\ell(I)+1$.
\end{conj}

Let $G$ be a bipartite graph. As we mentioned above $J(G)$ is a normal ideal. Thus, every power of $J(G)$ is integrally closed. Therefore, Theorem \ref{main} shows that Conjecture \ref{conje} is true for the powers of cover ideal of bipartite graphs.






\begin{thebibliography}{10}

\bibitem {abcj} A. M. Duval, B. Goeckner, C. J. Klivans, J. L. Martin, A non-partitionable Cohen-Macaulay simplicial complex, preprint.

\bibitem {b} M. Brodmann, The asymptotic nature of the analytic spread,
    {\it Math. Proc. Cambridge Philos. Soc.} {\bf 86} (1979), no. 1,
    35--39.

\bibitem {b'} L. Burch, Codimension and analytic spread, {\it Proc.
    Cambridge Philos. Soc.} {\bf 72} (1972), 369--373.

\bibitem {cpsty} A. Constantinescu, M. R. Pournaki, S. A. Seyed Fakhari, N.Terai, S. Yassemi, Cohen-Macaulayness and limit behavior of depth for powers of cover ideals, {\it Comm. Algebra}, {\bf 43} (2015), no. 1, 143--157.

\bibitem {cv} A. Constantinescu, M. Varbaro, Koszulness, Krull dimension,
    and other properties of graph-related algebras, {\it J. Algebraic
    Combin.} {\bf 34} (2011), no. 3, 375--400.

\bibitem{c} M. Cimpoea\c{s}, Several inequalities regarding Stanley depth, {\it Romanian Journal of Math.
    and Computer Science} {\bf 2}, (2012), 28--40.

\bibitem {c1} M. Cimpoea{\c{s}}, Stanley depth of monomial ideals with small number of generators,
    {\it Central European Journal of Mathematics}, {\bf 7} (2009), 629--634.

\bibitem {dhs} H. Dao, C. Huneke, J. Schweig, Bounds on the regularity and projective dimension of ideals associated to graphs, {\it J. Algebraic Combin.} {\bf 38} (2013), 37--55.

\bibitem {grv} I. Gitler, E. Reyes, R. H. Villarreal, Blowup algebras of
    ideals of vertex covers of bipartite graphs, {\it Contemp. Math.} {\bf
    376} (2005), 273--279.

\bibitem {hv} H. T. H${\rm \grave{a}}$, A. Van Tuyl, Monomial ideals, edge ideals of hypergraphs, and their graded Betti numbers, {\it J. Algebraic Combin.}
    {\bf 27} (2008), 215--245.

\bibitem {h} J. Herzog, A survey on Stanley depth. In "Monomial Ideals, Computations and Applications", A. Bigatti,
    P. Gim${\rm\acute{e}}$nez, E. S${\rm\acute{a}}$enz-de-Cabez${\rm\acute{o}}$n (Eds.), Proceedings of MONICA 2011. Lecture Notes in
    Math. {\bf 2083}, Springer (2013).

\bibitem {hh} J. Herzog, T. Hibi, {\it Monomial Ideals}, Springer-Verlag,
    2011.

\bibitem {hvz} J. Herzog, M. Vladoiu, X. Zheng, How to compute the Stanley
    depth of a monomial ideal, {\it J. Algebra} {\bf 322} (2009), no. 9,
    3151--3169.

\bibitem {k} M. Kummini, Regularity, depth and arithmetic rank of bipartite edge ideals, {\it J. Algebraic Combin.} {\bf 30} (2009), 429--445.

\bibitem {n} E. Nevo, Regularity of edge ideals of $C_4$-free graphs via the topology of the lcm-lattice, {\it J. Combin. Theory Ser. A} {\bf 118}
(2011), 491--501.

\bibitem {p'} I. Peeva, {\it Graded syzygies}, Algebra and Applications, vol. 14, Springer-Verlag London Ltd., London, 2011.

\bibitem{p} D. Popescu, Bounds of Stanley depth, {\it An. St. Univ. Ovidius. Constanta}, 19(2),(2011), 187--194.

\bibitem {psty} M. R. Pournaki, S. A. Seyed Fakhari, M. Tousi, S. Yassemi,
    What is $\ldots$ Stanley depth? {\it Notices Amer. Math. Soc.} {\bf 56}
    (2009), no. 9, 1106--1108.

\bibitem {r} A. Rauf, Stanley decompositions, pretty clean filtrations and
    reductions modulo regular elements, {\it Bull. Math. Soc. Sci. Math.
    Roumanie (N.S.)} {\bf 50(98)} (2007), no. 4, 347--354.

\bibitem {s3} S. A. Seyed Fakhari, Stanley depth and symbolic powers of monomial ideals, {\it Math. Scand.}, to appear.

\bibitem {s1} S. A. Seyed Fakhari, Stanley depth of the integral closure of
    monomial ideals, {\it Collect. Math.} {\bf 64} (2013), 351--362.

\bibitem {s} R. P. Stanley, Linear Diophantine equations and local
    cohomology, {\it Invent. Math.} {\bf 68} (1982), no. 2, 175--193.

\bibitem {va} A. Van Tuyl, Sequentially Cohen-Macaulay bipartite graphs:  vertex decom
posability and regularity, {\it Arch. Math. (Basel)} {\bf 93} (2009), 451--459.

\bibitem {w} R. Woodroofe, Matchings, coverings, and Castelnuovo-Mumford regularity, {\it J. Commut. Algebra} {\bf 6} (2014), 287--304.
\end{thebibliography}
\end{document}